\documentclass[conference,letterpaper]{IEEEtran}
\usepackage{comment}
\newcommand{\cderivative}{L}
\newcommand{\wsep}{w_{\mathrm{sep}}}

\usepackage[utf8]{inputenc} 
\usepackage[T1]{fontenc}
\usepackage{url}
\usepackage{ifthen}
\usepackage{cite}
\usepackage{useful_math}
\begin{document}
\title{Detecting Mutual Excitations in Non-Stationary Hawkes Processes} 

% %%% Single author, or several authors with same affiliation:
% \author{%
%  \IEEEauthorblockN{Author 1 and Author 2}
% \IEEEauthorblockA{Department of Statistics and Data Science\\
%                    University 1\\
 %                   City 1\\
  %                  Email: author1@university1.edu}% }

%%% Several authors with up to three affiliations:
\author{%
  \IEEEauthorblockN{Elchanan Mossel}
  \IEEEauthorblockA{Department of Mathematics \\
                    Massachusetts Institute of Technology\\
                    Cambridge, MA\\
                    Email: elmos@mit.edu}
\and
  \IEEEauthorblockN{Anirudh Sridhar}
  \IEEEauthorblockA{Department of Electrical \& Computer Engineering \\
                    New Jersey Institute of Technology\\
                    Newark, NJ\\
                    Email: anirudh.sridhar@njit.edu}
}

\maketitle

\begin{abstract}
We consider the problem of learning the network of mutual excitations (i.e., the dependency graph) in a non-stationary, multivariate Hawkes process. 
We consider a general setting where baseline rates at each node are time-varying and delay kernels are not shift-invariant. 
Our main results show that if the dependency graph of an $n$-variate Hawkes process is sparse (i.e., it has a maximum degree that is bounded with respect to $n$), our algorithm accurately reconstructs it from data after observing the Hawkes process for $T = \poly \log(n)$ time, with high probability.
Our algorithm is computationally efficient, and provably succeeds in learning dependencies even if only a subset of time series are observed and event times are not precisely known.
\end{abstract}

\section{Introduction}

Many large-scale systems exhibit self-exciting behavior, where events that occur at one site can trigger increased activity at other locations.
Hawkes processes are the classical model for self-exciting of this type, and have found applications in numerous domains, including epidemiology \cite{meyer2012spacetime, meyer2014powerlaw, dong2023nonstationary, chiang2022hawkes}, geophysics \cite{ogata1988statistical, ogata1999seismicity}, criminology \cite{Mohler2011crime, mohler2014marked} and finance \cite{bacry2015hawkes}. 
Understanding the mechanisms that drive excitations are critical for the prediction and control of these systems.

Hawkes processes are composed of two main elements: a baseline rate of events at each site or node, and delay kernels that determine the impact of past events on the rate of future events at the same site or node (self-excitations) as well as the impact on neighboring nodes (mutual excitations).
The stationary case, where baseline events are constant over time and delay kernels are shift-invariant, has received extensive mathematical treatment; see \cite{bacry2015hawkes} and references therein.
However, in many applications of interest, non-stationarities play a major role. For instance, the time of year can significantly impact the baseline spreading rate of a disease in a population. Furthermore, infections that occur at later points in time may have a reduced effect on disease dynamics as herd immunity increases or mitigation measures are implemented.
Without properly accounting for such non-stationarities, it may be unclear whether bursts in activity are due to changing environments or causal effects from other sites. 

\emph{Contributions.} In this paper, we develop a new algorithm with provable guarantees for detecting sparse causal interactions (i.e., mutual excitations) in non-stationary, multivariate Hawkes processes.
Though methods for estimating non-stationary Hawkes processes exist \cite{cai2022network, chen2013inference, roueff2016locally, chen2023detection, cheng2025spatiotemporal},
to the best of our knowledge this is the first work to provide theoretical guarantees for learning aspects of multivariate Hawkes processes with time-varying baseline rates and non-shift-invariant delay kernels.

The idea behind our approach is that when an event happens at site $i$, there is an abrupt jump in the event rates at sites that are influenced by $i$. 
To measure whether $i$-events cause such a jump in the event rate at site $j$, we can compare the number of times a $j$-event occurs shortly after an $i$-event (which we refer to as an $ij$-event) to the number of $ji$-events (where $j$ occurs shortly before $i$). 
If there is a significant difference between the two counts, there must exist interactions between $i$ and $j$.
However, bi-directional interactions between $i$ and $j$ can cause certain cancellations to happen, making it appear that the two processes are uncorrelated when in fact they are driven by each other. 
To address this limitations, we consider \emph{higher-order interactions} between $i$ and $j$; namely, we count events of the type $iij, iji, jii$ to assess how the rate of $j$-events is impacted by \emph{two} $i$-events. Combined, we show that pairwise and higher-order event counts conclusively determine whether or not $i$ and $j$ interact.
Our theoretical results show that to learn all interactions between $n$ interacting time series, it suffices to observe the system for $T = \poly \log(n)$ time.
Our method is simple to implement and enjoys additional robustness properties: interactions can be detected even if only a subset of sites are observed, and the algorithm only requires quantized (rather than exact) event times. 

\emph{Related work.} Hawkes processes were first formulated by A.G. Hawkes in 1971 \cite{hawkes1971}, and parameter estimation in this model has been a topic of intense study since then; see, e.g., \cite{reinhart2018review, cheng2025spatiotemporal} and references therein for a review of traditional and modern techniques.
While most theoretical work establishes the asymptotic consistency of certain estimators, the works \cite{hansen2015lasso, bacry2020sparse, wang2024statistical, cai2022network} additionally develop finite-sample guarantees. In particular, it is shown in various contexts that if the time series are observed for $T = \poly \log (n)$ time, the individual rate functions can be estimated with small error. 
Of these, only \cite{cai2022network} considers non-stationary Hawkes processes, though they require delay kernels to be shift invariant, which is not needed in the present work. 
However, we focus on learning the fundamental dependency structure, while \cite{cai2022network} infers all model parameters.

We also note that our work contributes to the literature on {robustly} learning Hawkes processes under event time uncertainty \cite{Kirchner2017estimation, trouleau19synchronization, trouleau21cumulants, cheng2024point, potiron2025mutually} and unobserved nodes \cite{liu2020partially, wang2021causal, thams2024local}. 

Beyond Hawkes processes, our work is closely related to the literature on learning graphical models from dynamics, such as the Glauber dynamics \cite{bresler2017learning, dutt2021exponential, gaitonde2024unified, gaitonde2024efficiently} or spreading processes \cite{NS12_cascades, ACFKP13_trace_complexity, WT04_epidemic, AA05_blog}.
Like the current work, a common strategy in this literature is to analyze the statistics of short sequences of events. We therefore conjecture that our methods may be broadly applicable.

\emph{Notation.} For a positive integer $n$, we denote $[n] : = \{ 1, \ldots, n \}$.
For a time-varying stochastic process, we let $\{ \cF_t \}_{t \ge 0}$ be the natural filtration.
We use standard asymptotic notation, i.e., $O(\cdot), \Omega(\cdot), o(\cdot)$, and all limits are as $n \to \infty$ unless otherwise specified. 

\emph{Organization.} In Section \ref{sec:models}, we define the multivariate Hawkes process and state our main results on learning the dependency graph of mutual excitations. Section \ref{sec:alg} provides the intuition and details behind our algorithm.
Section \ref{sec:alg_analysis} contains the proof of our main results, and we conclude in Section \ref{sec:conclusion}.
Additional proofs can be found in the Appendix.

\section{Models and results}
\label{sec:models}

\subsection{The multivariate Hawkes process}
\label{subsec:hawkes}

For a point process $N(t)$, its rate function is given by 
\[
\lambda(t) : = \lim_{\epsilon \to 0} \frac{\E [ N(t + \epsilon) \vert \cF_t] - N(t)}{\epsilon}.
\]

In (multivariate) Hawkes processes, the rate function can depend on the full history of the process itself as well as other point processes \cite{hawkes1971}. Formally, let $N_1(t), \ldots, N_n(t)$ denote a collection of point processes. The rate function of $N_i(t)$ is given by 
\begin{equation}
\label{eq:multivariate_hawkes}
\lambda_i(t) : = \mu_i(t) + \sum_{j = 1}^n w_{ij} \int_0^t \phi_{ij}(t , s) dN_j(s),
\end{equation}
where $\mu_i(t)$ is a non-negative deterministic function that captures the baseline rate of events, $\phi_{ij}$ is a non-negative delay kernel that captures how past events at $j$ affect the current rate of $i$-events, and $w_{ij}$ is a non-negative parameter that captures the magnitude of the mutual excitations of $j$ on $i$. Throughout the paper, we assume the following: 

\begin{assumption}[Baseline rates]
\label{as:baseline}
There is a constant $\mu_{\min} > 0$ such that $\mu_i(t) \ge \mu_{\min}$ for all $i \in [n]$ and $t \ge 0$.
\end{assumption}

\begin{assumption}[Delay kernels]
\label{as:delay}
For all $i,j \in [n]$, $\phi_{ij}$ is continuous and non-negative.
Furthermore, for all $s \ge 0$, $\phi_{ij}(s,s) = 1$ and $\phi_{ij}(\cdot, s)$ is decreasing. 
\end{assumption}

\begin{assumption}[Stability]
\label{as:stability}
There is a constant $\Phi > 0$ such that for all $i,j \in [n]$ and all $t \ge 0$, $\int_0^t \phi_{ij}(t,x) dx \le \Phi$.
Furthermore, there is a constant $\eta > 0$ satisfying $\sum_{j = 1}^n w_{ij} \int_0^t \phi_{ij}(t,x) dx \le 1 - \eta$ for all $i \in [n]$ and all $t \ge 0$.
\end{assumption}

Assumptions \ref{as:baseline}, \ref{as:delay} and \ref{as:stability} are quite standard, although it is typically assumed in theoretical work that $\mu_i$ is constant and $\phi_{ij}(t, s)$ is a function of $t - s$.
Assumption \ref{as:stability} ensures that the process does not blow up in finite time. In particular, if $\phi_{ij}(t,x) = \varphi_{ij}(t - x)$ for some function $\varphi_{ij}$ as is commonly assumed, the condition in the assumption reduces to $\sum_{j = 1}^n w_{ij} \int_0^\infty \varphi_{ij}(x) dx \le 1 - \eta$, which ensures that the maximum rate of events is almost surely finite; see, e.g., \cite[Theorem 7]{bremaud1996stability}.

\begin{assumption}[Smoothness]
\label{as:smoothness}
There is a constant $\cderivative$ such that for all $i \in [n]$ and $t \ge 0$, $| \mu_i'(t) | / \mu(t) \le \cderivative$.
Additionally, for all $i,j \in [n]$ and $t \ge s \ge 0$, $| \frac{d\phi_{ij}}{dt} (t, s) | / \phi_{ij}(t, s) \le \cderivative$.
\end{assumption}

Assumption \ref{as:smoothness} implies that $\mu_i(\cdot)$ and $\phi_{ij}( \cdot, s)$ do not explode or decay faster than an exponential function. 
This assumption is satisfied in the most common setting for Hawkes processes, where $\mu_i$ is constant and $\phi_{ij}$ is an exponentially decaying function. However, we remark that heavy-tailed delay kernels such as those used in earthquake modeling \cite{ogata1988statistical, ogata1999seismicity} do not satisfy this assumption. 
An important avenue for future work is to remove this assumption to increase the applicability of our results.

\begin{assumption}[Mutual excitations]
\label{as:mutual}
There are constants $w_{\min}, w_{\max} > 0$ such that, for all $i,j \in [n]$, either (1) $w_{ij} = 0$ or (2) $w_{\min} \le w_{ij} \le w_{\max}$.
Additionally, there is a constant $\wsep >0$ such that $w_{ii} - w_{ij}  \ge \wsep$ for all distinct pairs $i,j \in [n]$.
\end{assumption}

The assumption that $w_{ij} = 0$ or $w_{ij} \ge w_{\min}$ is a natural one in the literature on learning graphical models, as it allows the algorithm to set a concrete threshold to determine the presence of an edge.
The assumption that $w_{ij} \le w_{\max}$ is helpful in bounding the maximum rate of events, and is useful in our analysis but is not needed by our algorithm.
Additionally, the condition $w_{ii} > w_{ij}$ reflects a locality constraint, as self-excitations may be stronger than influences from other locations.
Although not strictly necessary for Hawkes process inference, this assumption prevents degenerate edge cases in our algorithm.
We finally note that our results remain valid under the more general condition $| w_{ii} - w_{ij} | \ge \wsep$.

\begin{assumption}[Sparsity]
\label{as:sparsity}
For each $i \in [n]$, there are at most $d$ nonzero elements in the collection $\{ w_{ij} \}_{j \in [n]}$.
\end{assumption}

Common in the literature on graphical models, the sparsity assumption reflects that influential connections are typically few relative to the total system size.

\subsection{Detecting mutual excitations}

Our goal is to detect whether there exist mutual excitations between nodes $i$ and $j$ (i.e., whether $w_{ij} > 0$ or $w_{ji} > 0$), for any pair $i,j \in [n]$.
The set of all ground-truth mutual excitations can be captured by an undirected dependency graph $G^*$, also known as the \emph{Markov blanket}, which has vertex set $[n]$, and $(i,j)$ is an edge in $G^*$ if and only if $w_{ij} > 0$ or $w_{ji} > 0$.
Our main result shows that $G^*$ can be learned after observing the multivariate Hawkes process for $T = \poly \log (n)$ time.

\begin{theorem}
Suppose that the point processes $N_1, \ldots, N_n$ are observed in $[0,T]$, with $T = \log^{100}(n)$. Then there is an estimator $\widehat{G}$ such that $\widehat{G} = G^*$ with probability $1 - o(1)$.
\end{theorem}

We remark that observing the system for $T= \poly \log (n)$ time is a common requirement for accurate parameter estimation in Hawkes processes \cite{hansen2015lasso, bacry2020sparse, wang2024statistical, cai2022network}.
Here, we establish the same guarantee in a more general non-stationary setting, though we focus only on recovering $G^*$ rather than full parameter estimation.
Details about the estimator $\widehat{G}$ are in Section \ref{sec:alg}.

\section{Estimation algorithm}
\label{sec:alg}

\subsection{Analysis of pairwise events}
\label{subsec:pairwise_analysis}

Suppose we wish to detect whether there are mutual excitations between nodes $i$ and $j$.
A simple but crucial insight is that if $w_{ji} > 0$, each event at $i$ triggers an immediate jump in the rate of events at $j$. 
That is, if an $i$-event occurs at time $t$, then
\begin{equation}
\label{eq:lambda_j_jump}
\lambda_j(t) - \lambda_j(t^-) = w_{ji}.
\end{equation}
Due to this excitation, $j$-events occur with an elevated probability shortly after $i$-events (we refer to such a sequence of events as an $ij$-event). 
In particular, one may expect that $ij$ events occur much more frequently than $ji$ events (where a $j$-event occurs shortly before an $i$-event).

Unfortunately, a rigorous analysis of this idea shows that it fails in general to detect mutual excitations.
To formalize the approach, we define the indicator random variables 
\begin{align*}
X_{ij}(t, \epsilon) & : = \mathbf{1}( N_i[t, t + \epsilon] = 1, N_j[t + \epsilon, t + 2 \epsilon] = 1) \\
X_{ji}(t, \epsilon) & : = \mathbf{1} ( N_j[t, t + \epsilon] = 1, N_i[t + \epsilon, t + 2 \epsilon] = 1).
\end{align*}
The first indicator represents the possibility of observing an $ij$-event at time $t$, and the second represents the possibility of observing a $ji$-event at time $t$. 
The statistic of interest to us is 
\[
\Delta_{ij}^{(1)} (t, \epsilon) : = X_{ij}(t, \epsilon) - X_{ji}(t, \epsilon).
\]
It is useful to think of $\Delta_{ij}^{(1)}$ as the \emph{first-order derivative of $j$ with respect to $i$}, as it captures the difference in cases where a $j$-event occurs before and after an $i$-event. 

Under the assumption in Section \ref{subsec:hawkes}, we show that 
\begin{align}
\label{eq:Xij_explained}
\E [ X_{ij}(t, \epsilon) \vert \cF_t] & \approx \epsilon^2 \lambda_i(t) ( \lambda_j(t) + w_{ji}) \\
\label{eq:Xji_explained}
\E [ X_{ji}(t, \epsilon) \vert \cF_t] & \approx \epsilon^2 \lambda_j(t) ( \lambda_i(t) + w_{ij}),
\end{align}
where $\approx$ hides additive terms of lower order in $\epsilon$, and $\{ \cF_t \}_{t \ge 0}$ is the natural filtration.
Equation \eqref{eq:Xij_explained} can be explained as follows. The probability that an $i$-event occurs in $[t, t + \epsilon]$ is approximately $\epsilon \lambda_i(t)$ for sufficiently small $\epsilon$, and similarly, a $j$-event occurs in $[t + \epsilon, t + 2 \epsilon]$ with probability $\epsilon \lambda_j(t + \epsilon)$. However, due to the influence of the $i$-event on $j$ (as described in \eqref{eq:lambda_j_jump}), $\lambda_j(t + \epsilon )\approx \lambda_j(t) + w_{ji}$. Multiplying the two probabilities yields \eqref{eq:Xij_explained}. 
Note that this logic does not hold if there is a $k$-event in $[t, t + \epsilon]$ that also affects $\lambda_j(t)$. However, this would require three events to occur within a $O(\epsilon)$ time window (namely, $i,j,k$), which has a probability of lower order than $\epsilon^2$ if excitations are sparse (Assumption \ref{as:sparsity}).
Equation \eqref{eq:Xji_explained} follows from similar arguments.
For more details, see Lemma \ref{lemma:EX} in Section \ref{sec:alg_analysis}.

In light of \eqref{eq:Xij_explained} and \eqref{eq:Xji_explained}, we have 
\begin{equation}
\label{eq:Delta1_expectation}
\E [ \Delta_{ij}^{(1)}(t, \epsilon)  \vert \cF_t]/\epsilon^2  \approx  w_{ji} \lambda_i(t) - w_{ij} \lambda_j(t) .
\end{equation}
If neither process influences each other, i.e., $w_{ij} = w_{ji} = 0$, then $\Delta_{ij}^{(1)}$ has zero mean, reflecting that $ij$-events and $ji$-events occur with similar probability.
If $w_{ji} > 0$ while $w_{ij} = 0$, then the expectation reduces to $ w_{ji} \lambda_i(t)$, indicating that the influence of $i$ on $j$ increases the chance of an $ij$-event compared to $ji$-events. 
However, if $w_{ij}$ and $w_{ji}$ are both positive, there can be a confounding effect, where the probabilities of $ij$-events and $ji$-events \emph{both} increase, while their difference remains small. 
Consequently, counting $ij$-events and $ji$-events is not enough to conclusively determine mutual excitations.

\subsection{Analysis of event triples}

We show that the ambiguities of pairwise event counts can be fully resolved by counting triples of events. In particular, we study how $j$ reacts to \emph{two} $i$-events, rather than just one. There are three ways in which one $j$-event and two $i$-events can be ordered, namely $jii, iji$ and $iij$, which capture cases where $j$-events are affected by zero, one or two $i$-events. 
We proceed by analyzing indicator variables corresponding to each type of event sequence.
We let
\begin{align*}
X_{jii}(t, \epsilon) & : = \mathbf{1}( N_j[t, t + \epsilon] = 1, N_i[t + \epsilon, t + 2 \epsilon] = 1, \\
& \hspace{1cm} N_i[t + 2 \epsilon, t + 3 \epsilon] = 1 ),
\end{align*}
with similar definitions for $X_{iji}(t, \epsilon)$ and $X_{iij}(t, \epsilon)$.
We also define 
\begin{align*}
x_{jii}(t, \epsilon) & : = \lambda_j(t) ( \lambda_i(t) + w_{ij}) ( \lambda_i(t) + w_{ij} + w_{ii}) \\
x_{iji}(t, \epsilon) & : = \lambda_i(t) ( \lambda_j(t) + w_{ji} ) ( \lambda_i(t) + w_{ii} + w_{ij}) \\
x_{iij}(t, \epsilon) & : = \lambda_i(t) ( \lambda_i(t) + w_{ii}) ( \lambda_j(t) + 2 w_{ji}).
\end{align*}

Following similar reasoning as the derivation of \eqref{eq:Xij_explained} and \eqref{eq:Xji_explained}, we can show that $\E [ X_{jii}(t, \epsilon ) \vert \cF_t] \approx \epsilon^3 x_{jii}(t, \epsilon)$, $\E [ X_{iji}(t, \epsilon) \vert \cF_t] \approx \epsilon^3 x_{iji}(t, \epsilon)$, and $\E [ X_{iij}(t, \epsilon) \vert \cF_t] \approx \epsilon^3 x_{iij}(t, \epsilon)$ where, as before, $\approx$ hides additive terms of lower order in $\epsilon$ (see Lemma \ref{lemma:EX} in Section \ref{sec:alg_analysis} for more details). 
We then define
\begin{equation*}
\Delta_{ij}^{(2)}(t, \epsilon) : = X_{iij}(t, \epsilon) - 2 X_{iji}(t, \epsilon) + X_{jii}(t, \epsilon).
\end{equation*}
This statistic is evocative of a \emph{second derivative}: it aims to capture the impact of $j$-events interacting with two $i$-events, while subtracting out the impact of $ii$, $ij$ or $ji$ events.
Additionally, it follows readily that
\begin{equation}
\label{eq:Delta2_expectation}
\frac{\E [ \Delta_{ij}^{(2)}(t, \epsilon) \vert \cF_t]}{\epsilon^3}  \approx w_{ij} ( w_{ii} + w_{ij}) \lambda_j(t) - 2 w_{ij} w_{ji} \lambda_i(t).
\end{equation}
Importantly, \eqref{eq:Delta2_expectation} shows that the conditional expectation of $\Delta_{ij}^{(2)}$ is also a linear combination of $\lambda_i(t)$ and $\lambda_j(t)$, which leads to a convenient joint characterization of the first and second derivatives. Indeed, defining the matrix 
\[
M : = \begin{pmatrix}
w_{ji} & - w_{ij} \\
-2 w_{ij} w_{ji} & w_{ij} (w_{ii} + w_{ij})
\end{pmatrix},
\]
equations \eqref{eq:Delta1_expectation} and \eqref{eq:Delta2_expectation} lead to 
\begin{equation}
\label{eq:Delta_matrix_equation}
\begin{pmatrix}
\E [ \Delta_{ij}^{(1)}(t, \epsilon) \vert \cF_t] / \epsilon^2 \\
\E [ \Delta_{ij}^{(2)}(t, \epsilon)  \vert \cF_t] / \epsilon^3
\end{pmatrix} 
\approx 
M 
\begin{pmatrix}
    \lambda_i(t) \\
    \lambda_j(t)
\end{pmatrix} = : \mathbf{v}.
\end{equation}

Let us now revisit the case where $w_{ij}$ and $w_{ji}$ are both positive, which may cause ambiguities in pairwise event counts.
It can be seen that $\det(M) = w_{ij} w_{ji} (w_{ii} - w_{ij}) > 0$ by Assumption \ref{as:mutual}, hence $M$ is non-degenerate. Thus, since $\lambda_i(t)$ and $\lambda_j(t)$ are both positive, it follows that $\| v \|_2 > 0$, meaning that at least one entry of $v$ is bounded away from zero. Consequently, the conditional expectation of $\Delta_{ij}^{(2)}(t,\epsilon)$ is able to resolve any ambiguities left open by $\Delta_{ij}^{(1)}(t,\epsilon)$. 

\subsection{Algorithm description}

The insights of the previous subsections show that for any $t \ge 0$, the conditional expectations of $\Delta_{ij}^{(1)}(t, \epsilon)$ and $\Delta_{ij}^{(2)}(t, \epsilon)$ can \emph{jointly} reveal whether there exist mutual excitations between $i$ and $j$. To extract this information empirically, we sum over many values of $t$ in order to boost the signal while reducing the relative variance.
Formally, for a time horizon $T \ge 0$ and $\epsilon > 0$, we define 
\begin{align}
\label{eq:Dij1}
D_{ij}^{(1)}(T, \epsilon) & : = \sum_{k = 0}^{ \lfloor \frac{T}{3 \epsilon} \rfloor} \Delta_{ij}^{(1)}(k\epsilon, \epsilon) \\
\label{eq:Dij2}
D_{ij}^{(2)} (T, \epsilon) & : = \sum_{k = 0}^{\lfloor \frac{T}{3 \epsilon} \rfloor} \Delta_{ij}^{(2)}(k\epsilon, \epsilon).
\end{align}
In the statistics above, we sum over time increments separated by $3\epsilon$ so that the $\Delta_{ij}^{(2)}(k\epsilon,\epsilon)$'s do not contain overlapping information. 
Algorithm \ref{alg:structure_learning} below shows how these statistics are used to detect mutual excitations.

\begin{breakablealgorithm}
\caption{Detecting mutual excitations}
\label{alg:structure_learning}
\begin{algorithmic}[1]
\Require{Point processes $N_1, \ldots, N_n$ observed in the interval $[0,T]$, parameter $\epsilon > 0$, threshold $h > 0$.}
\Ensure{A graph $\widehat{G}$ with vertex set $[n]$.}
\State Initialize $\widehat{G}$ to be the empty graph on $[n]$.
\State For each pair of distinct $i,j \in [n]$, add $(i,j)$ to $\widehat{G}$ if 
\begin{equation}
\label{eq:statistic_sum}
\frac{|D_{ij}^{(1)}(T, \epsilon)|}{T \epsilon} + \frac{|D_{ij}^{(2)}(T, \epsilon)|}{T \epsilon^2} \ge h.
\end{equation}
\State Return $\widehat{G}$.
\end{algorithmic}
\end{breakablealgorithm}

The following result shows that with the right values of $(T, \epsilon, h)$, the algorithm returns the correct answer with high probability. The proof can be found in Section \ref{sec:alg_analysis}.

\begin{theorem}
\label{thm:alg}
If $T = \log^{100}(n), \epsilon = \log^{-17}(n)$ and $h = w_{\min} \wsep \mu_{\min} / 8$, then Algorithm \ref{alg:structure_learning} outputs $\widehat{G}$ satisfying $\widehat{G} = G^*$ with probability $1 - o(1)$.
\end{theorem}

We make a few remarks about Algorithm \ref{alg:structure_learning} and Theorem \ref{thm:alg}.
Unlike the prior work \cite{cai2022network, chen2013inference, roueff2016locally, chen2023detection}, our algorithm directly estimates $G^*$ without needing to estimate baseline rates and delay kernels. 
Our algorithm offers two additional strengths: it recovers the correct subgraph when only a subset of nodes is observed since the edge $(i, j)$ depends solely on $N_i$ and $N_j$, and its use of quantized event times ensures robustness against small adversarial perturbations within $\epsilon$-sized intervals.

\section{Analysis of Algorithm \ref{alg:structure_learning}}
\label{sec:alg_analysis}

In this section, we prove our main result, Theorem \ref{thm:alg}. To do so, we state and explain a few technical building blocks. 
The first supporting result characterizes the expectation of indicators corresponding to short waiting times. 
To state the result, for $i,j \in [n]$ we define $x_{ij}(t, \epsilon) : = \lambda_i(t) ( \lambda_j(t) + w_{ji})$, and for (not necessarily distinct) $i,j,k \in [n]$ we define $x_{ijk}(t, \epsilon) : = \lambda_i(t) ( \lambda_j(t) + w_{ji}) ( \lambda_k(t) + w_{ki} + w_{kj} )$.

\begin{lemma}
\label{lemma:EX}
Let $t \ge 0$, let $\epsilon \in (0,1]$, and assume that $\lambda_{\max}(t) \le \Lambda$.
There is a sufficiently large constant $C > 0$ depending on $\mu_{\max}, w_{\max}$ and $\cderivative$ such that if $d \Lambda \epsilon \le 1 / C, \epsilon \ge n^{-1/4}$ and $\Lambda \ge C d \log n$, then the following hold:
\begin{align}
\label{eq:xij}
\left| \E[X_{ij}(t, \epsilon) \vert \cF_t] - \epsilon^2 x_{ij}(t, \epsilon) \right|& \le C ( d \Lambda \epsilon)^3, \\
\label{eq:xijk}
\left| \E [ X_{ijk}(t, \epsilon) \vert \cF_t] - \epsilon^3 x_{ijk}(t, \epsilon) \right| & \le C ( d \Lambda \epsilon)^4.
\end{align}
\end{lemma}

The intuition behind the result can be found in Section \ref{subsec:pairwise_analysis} (see in particular the discussion following equations \eqref{eq:Xij_explained} and \eqref{eq:Xji_explained}). Proof details can be found in Appendix \ref*{sec:short_waiting_times}.
%of \cite{mossel2026hawkes}.

Crucially, the error term in Lemma \ref{lemma:EX} depends on $\lambda_{\max}(t)$. In our next result, we establish an upper bound for this quantity over the entire observation window.

\begin{lemma}
\label{lemma:maximum_rate}
With probability tending to 1 as $nT \to \infty$, it holds that $\sup_{0 \le t \le T} \lambda_{\max}(t) \le d^2 \log^4 (nT)$.
\end{lemma}

Proving the lemma is technically quite involved since the rate functions themselves are stochastic and depend on the full history of events. Our proof identifies key martingales relevant to $\lambda_{\max}$ and carefully bounds their maximal fluctuations over $[0,T]$. 
Full proof details can be found in Appendix \ref*{subsec:maximum_rate}.
%of \cite{mossel2026hawkes}.

Together, Lemmas \ref{lemma:EX} and \ref{lemma:maximum_rate} allow us to establish concentration results for the statistics $D_{ij}^{(1)}(T, \epsilon)$ and $D_{ij}^{(2)}(T, \epsilon)$. 
In the following result, we define as shorthand $\Lambda_i : = \sum_{k = 0}^{\lfloor T / (3 \epsilon) \rfloor} \lambda_i(k \epsilon)$, with a similar definition for $\Lambda_j$.

\begin{lemma}
\label{lemma:D_analysis}
Let $T > 0$ and $\epsilon \in (0,1]$. Additionally, assume that $\epsilon > 0$ satisfies $n^{-1/4} \le \epsilon \le (Cd^3 \log^4 (nT))^{-1}$,
where $C > 0$ is a sufficiently large constant depending on $\mu_{\max}$ and $w_{\max}$. Then with probability tending to 1 as $n \to \infty$, it holds for all $i,j \in [n]$ that
\begin{multline}
\label{eq:D1_characterization}
\left| D_{ij}^{(1)}(T, \epsilon) - \epsilon^2 ( w_{ji} \Lambda_i - w_{ij} \Lambda_j ) \right| \\
\lesssim \sqrt{\frac{T \log n}{\epsilon}} + d^9 T \log^{12}(nT) \epsilon^2,
\end{multline}
and that
\begin{multline}
\label{eq:D2_characterization}
\left| D_{ij}^{(2)}(T, \epsilon) - \epsilon^3 w_{ij} ( (w_{ij} + w_{ii}) \Lambda_j - 2 w_{ji} \Lambda_i) \right| \\
\lesssim \sqrt{ \frac{ T \log n}{\epsilon}} + d^{12} T \log^{16}(nT) \epsilon^3,
\end{multline}
where $\lesssim$ hides multiplicative factors that depend only on $\mu_{\max}, w_{\max}$ and $\cderivative$.
\end{lemma}

The proof of \eqref{eq:D1_characterization} proceeds in two main steps. First, we bound the absolute value of $D_{ij}^{(1)}(T, \epsilon) - \sum_{k = 0}^{ \lfloor T / (3 \epsilon) \rfloor} \E [ \Delta_{ij}^{(1)}(k\epsilon, \epsilon )\vert \cF_{k \epsilon}]$ through standard martingale concentration results (i.e., Azuma's inequality). Then, the sum of the conditional expectations is characterized by using Lemma \ref{lemma:EX}. Equation \eqref{eq:D2_characterization} is proved similarly. Full details can be found in Appendix \ref*{subsec:D_analysis}.
%of \cite{mossel2026hawkes}.

\begin{proof}[Proof of Theorem \ref{thm:alg}]
Our strategy is to show that \eqref{eq:statistic_sum} can only hold if at least one of $w_{ij}$ or $w_{ji}$ is positive.

We start by noting that for our choice of $T$ and $\epsilon$, the upper bound in \eqref{eq:D1_characterization} of Lemma \ref{lemma:D_analysis} is $o(T \epsilon)$ and the upper bound in \eqref{eq:D2_characterization} is $o(T \epsilon^2)$.
Consequently, if $w_{ij} = w_{ji} = 0$, Lemma \ref{lemma:D_analysis} implies that $| D_{ij}^{(1)} (T, \epsilon) | / (T \epsilon)$ and $| D_{ij}^{(2)}(T, \epsilon) | / (T \epsilon^2)$ are both $o(1)$, which shows that \eqref{eq:statistic_sum} cannot hold.

We next consider the case where $w_{ij} = 0$ and $w_{ji} \ge w_{\min}$.
Note that $\Lambda_i = \sum_{k = 0}^{\lfloor T / ( \epsilon) \rfloor} \lambda_i(k \epsilon) \ge \mu_{\min} \lfloor T / (3 \epsilon) \rfloor$,
where the inequality holds by Assumption \ref{as:baseline}, since $\lambda_i(t) \ge \mu_i(t) \ge \mu_{\min}$. The same lower bound holds for $\Lambda_j$.
Lemma \ref{lemma:D_analysis} then implies that $| D_{ij}^{(1)}(T, \epsilon) | / (T \epsilon) \ge \epsilon w_{ji} \Lambda_i / T - o(1) \ge w_{\min} \mu_{\min} / 4$, where the second inequality follows from the lower bound on $\Lambda_i$.
Equation \eqref{eq:statistic_sum} in this case therefore holds.

Finally, we consider the case where $w_{ij} \ge w_{\min}$. As a shorthand, let $\overline{\Lambda}_i : = \epsilon \Lambda_i / T$ and $\overline{\Lambda_j} : = \epsilon \Lambda_j / T$ denote normalized versions of $\Lambda_i$ and $\Lambda_j$, respectively, and also define $A := w_{ji} \overline{\Lambda}_i - w_{ij} \overline{\Lambda}_j$. Then by Lemma \ref{lemma:D_analysis}, 
\begin{align*}
& \frac{| D_{ij}^{(1)}(T, \epsilon)|}{T \epsilon} + \frac{| D_{ij}^{(2)}(T, \epsilon)|}{T \epsilon^2} \\
& \ge |A| + | w_{ij} ( (w_{ij} + w_{ii}) \overline{\Lambda}_j - 2 w_{ji} \overline{\Lambda}_i ) | - o(1) \\
& = |A| + 2 w_{ij} \left|  \frac{(w_{ii} - w_{ij}) \overline{\Lambda}_j}{2} - A \right| - o(1) \\
& \stackrel{(a)}{\ge} w_{\min} \left( |A| + \left| \frac{(w_{ii} - w_{ij} ) \overline{\Lambda}_j}{2} - A \right|  \right) - o(1) \\
& \stackrel{(b)}{\ge} \frac{w_{\min} (w_{ii} - w_{ij}) \overline{\Lambda}_j}{2} - o(1) \\
& \stackrel{(c)}{\ge} \frac{w_{\min} \wsep \mu_{\min}}{8}.
\end{align*}
Above, $(a)$ holds since $2 w_{ij} \ge w_{\min}$ and $w_{\min} \le 1$; $(b)$ follows from the triangle inequality and the positivity of $\overline{\Lambda}_j$; $(c)$ follows from the lower bound on $\Lambda_j$ we established earlier in the proof.
\end{proof}

\section{Conclusion}
\label{sec:conclusion}

In this work, we proposed a new algorithm for learning the dependency graph $G^*$ in non-stationary multivariate Hawkes processes.
We obtain provable guarantees for recovering $G^*$ in a more general setting than prior work, where baseline rates are time-varying and delay kernels are not shift-invariant.
Future work includes an investigation of higher-order event sequences to estimate $w_{ij}$, and the removal of Assumption \ref{as:smoothness} to broaden the algorithm's applicability.

\section*{Acknowledgment}
E.M. and A.S. acknowledge support from ARO MURI N000142412742, a Vannevar Bush Faculty Fellowship ONR-N00014-20-1-2826, NSF award CCF-1918421, and a Simons Investigator Award. 

\newpage
%%%%%%
%% To balance the columns at the last page of the paper use this
%% command:
%%
\enlargethispage{-1.2cm} 
%%
%% If the balancing should occur in the middle of the references, use
%% the following trigger:
%%
%\IEEEtriggeratref{7}
%%
%% which triggers a \newpage (i.e., new column) just before the given
%% reference number. Note that you need to adapt this if you modify
%% the paper.  The "triggered" command can be changed if desired:
%%
%\IEEEtriggercmd{\enlargethispage{-20cm}}
%%
%%%%%%

%%%%%%
%% References:
%% We recommend the usage of BibTeX:
%%
\bibliographystyle{IEEEtran}
\bibliography{references}
%%
%% where we here have assumed the existence of the files
%% definitions.bib and bibliofile.bib.
%% BibTeX documentation can be obtained at:
%% http://www.ctan.org/tex-archive/biblio/bibtex/contrib/doc/
%%%%%%

\clearpage

\appendix

\subsection{Additional notation}

For $i \in [n]$, we denote $\cN(i) : = \{ j : w_{ij} > 0 \}$.
For a set $S$, we denote $|S|$ to be its cardinality. We use the shorthand $a \lor b$ and $a \land b$ to represent the maximum and minimum of real numbers $a$ and $b$, respectively.

\subsection{Preliminary results}

We state and prove a few supporting results.
Our first result bounds the value that the maximum rate can take over short intervals. 

\begin{lemma}
\label{lemma:short_term_rate_bound}
Suppose that $\Lambda \ge C d \log n$, where $C= C(\mu_{\max}, w_{\max})$ is a sufficiently large constant. Then for any $t \ge 0$ and $\epsilon \ge 0$, it holds on the $\cF_t$-measurable event $\cE := \{ \lambda_{\max}(t) \le \Lambda \}$ that
\[
\p \left( \left. \sup_{s \in [t, t + \epsilon)} \lambda_{\max}(s) \ge 2 \Lambda  \right \vert \cF_t \right) \mathbf{1}(\cE) = o \left( \frac{1}{n} \right).
\]
\end{lemma}

\begin{proof}
Define the stopping time $\tau : = \inf \{ s \ge t : \lambda_{\max}(s) \ge 2 \Lambda \}$. 
Notice that if $\tau \le t + \epsilon$ and $\lambda_{\max}(t) \le \Lambda$, then there exists $i$ such that $\lambda_i(\tau) - \lambda_i(t) \ge \Lambda $. Additionally, since increases in $\lambda_i$ can only happen if $\mu_i$ increases or if an event in $\cN_i \cup \{i \}$ occurs, we have the bound
\begin{align*}
\lambda_i(\tau) - \lambda_i(t) & \le \mu_{\max} +  \sum_{j = 1}^n w_{ij}( N_j(\tau) - N_j(t)) \\
& \le \mu_{\max} + d w_{\max} \max_{j \in [n]} ( N_j(\tau) - N_j(t)),
\end{align*}
where the first inequality above follows since $\mu_i$ is non-negative, and since the impact of any event has a decreasing effect over time.
Consequently, if $\lambda_{\max}(t) \le \Lambda$ and $\tau \le t + \epsilon$, then 
\[
\max_{j \in [n]} ( N_j(\tau \land ( t + \epsilon)) - N_j(t) ) \ge \frac{\Lambda - \mu_{\max}}{d w_{\max}}.
\]
As a result, 
\begin{multline*}
\p ( \tau \le t + \epsilon \vert \cF_t ) \mathbf{1}( \lambda_{\max}(t) \le \Lambda) \\
\le \sum_{j = 1}^n \p \left( \left. N_j(\tau \land (t + \epsilon)) - N_j(t) \ge \frac{\Lambda - \mu_{\max}}{d w_{\max}} \right \vert \cF_t \right)
\end{multline*}
To bound the terms in the sum, notice that $N_j(\tau \land (t + \epsilon)) - N_j(t)$ is stochastically dominated by $Z \sim \mathrm{Poi}( 2 \Lambda  \epsilon)$. 
If $\Lambda \ge c d \log n$ for a sufficiently large positive constant $c$ depending on $\mu_{\max}$ and $w_{\max}$, then it can be shown through standard Poisson tail bounds, e.g., Bernstein's inequality, that each term is $o ( 1/n^2)$. The claimed result follows.
\end{proof}

The following lemma bounds the probability that many events occur in a short timeframe. 

\begin{lemma}
\label{lemma:multiple_firings}
Let $S \subseteq [n]$, $t \ge 0$, $\epsilon \in (0,1]$ and let $k \ge 1$ be an integer.
Additionally assume that $\Lambda |S| \epsilon \le 1/4$, $\epsilon \ge n^{-1/k}$ and that $\Lambda \ge C d \log n$ where $C = C(\mu_{\max}, w_{\max})$ is a sufficiently large constant.
Define the event 
\[
\cE : = \{ \text{At least $k$ events occur in $S$ in $[t, t + \epsilon)$} \}.
\]
Then $\p ( \cE \vert \cF_t ) \mathbf{1}(\lambda_{\max}(t) \le \Lambda ) \le (8  |S| \Lambda \epsilon )^k$. 
\end{lemma}

\begin{proof}
Let $N_S(t)$ be the point process counting the number of events in $S$, with a corresponding rate function $\lambda_S(t) : = \sum_{i \in S} \lambda_i(t)$. We also define the stopping time $\tau : = \inf \left \{ s \ge t :  \lambda_{\max}(s) \ge 2 \Lambda \right \}$.
Then we can bound
\begin{multline}
\label{eq:NS_probability_bound}
\p ( N_S(t + \epsilon) - N_S(t) \ge k \vert \cF_t )  \le \p ( \tau \le t + \epsilon \vert \cF_t ) \\
+ \p ( N_S(\tau \land (t + \epsilon) ) - N_S(t) \ge k \vert \cF_t ) .
\end{multline}
By Lemma \ref{lemma:short_term_rate_bound}, the first term on the right hand side is at most $1/n$, provided the event $\{ \lambda_{\max}(t) \le \Lambda \}$ holds. 
To handle the second term on the right hand side of \eqref{eq:NS_probability_bound}, notice that the stopped process $N_S(\cdot \land \tau)$ can be stochastically dominated by a Poisson process with rate $2 \Lambda |S|$. This implies in particular that $N_S(t + \epsilon \land \tau) - N_S(t)$ is stochastically dominated by $A \sim \mathrm{Poi}( 2\Lambda |S| \epsilon)$. Hence
\begin{align*}
& \p ( N_S((t + \epsilon) \land \tau) - N_S(t) \ge k \vert \cF_t )  \le \p ( A \ge k ) \\
& = \sum_{m = k}^\infty e^{- 2 \Lambda |S|\epsilon} \frac{(2 \Lambda |S| \epsilon)^m}{m!}  \le \sum_{m = k}^\infty (2 \Lambda |S| \epsilon)^m \\
& \le 2 (2 \Lambda |S| \epsilon)^k  \le ( 4 \Lambda |S| \epsilon)^k.
\end{align*}
Above, the second inequality follows provided $\Lambda |S| \epsilon \le 1/4$.
Substituting both bounds into \eqref{eq:NS_probability_bound} and using $\epsilon^k \ge 1/n$ yields the claimed result. 
\end{proof}

The next result characterizes the first event time of a general deterministic inhomogeneous point process. The proof relies on Taylor approximations.

\begin{lemma}
\label{lemma:T_deterministic_rate}
Let $r: \mathbb{R}_{\ge 0} \to \mathbb{R}_{\ge 0}$ be a smooth function satisfying the following properties for all $t \ge 0$:
\begin{enumerate}
    \item $0 \le r(t) \le r_{\max}$, for some $r_{\max} \ge 1$;
    \item $|r'(t) |/ r(t) \le \cderivative$, where $\cderivative$ is the same constant used in Assumptions \ref{as:baseline} and \ref{as:delay}.
\end{enumerate} 
Let $T$ be the first event time of an inhomogeneous point process driven by the rate function $r(t)$. Then 
\[
\left| \p ( T \le \epsilon ) - \epsilon r(0) \right| \le \cderivative r_{\max}^2 \epsilon^2.
\]
\end{lemma}

\begin{proof}
The probability density function of $T$ is given by 
\[
f(t) : = r(t) \exp \left( - \int_0^t r(s) ds \right).
\]
Our first goal is to derive error bounds for $|f(t) - r(0) |$ by a Taylor approximation. To this end, note that 
\[
f'(t) = ( r'(t) - r(t)^2) \exp \left( - \int_0^t r(s) ds \right).
\]
In particular, 
\[
| f'(t) | \le | r'(t) | + r(t)^2 \le \cderivative r_{\max} + r_{\max}^2 \le 2 \cderivative r_{\max}^2.
\]
Consequently, by Taylor's theorem, it holds for $t \in [0,\epsilon]$ that $|f(t) - r(0) | \le 2 \cderivative r_{\max}^2 t$.
We can then bound
\begin{multline*}
| \p ( T \le \epsilon ) - \epsilon r(0) |  \le \int_0^\epsilon |f(t) - r(0) | dt \\
\le 2 \cderivative r_{\max}^2 \int_0^\epsilon t dt  \le \cderivative r_{\max}^2 \epsilon^2.
\end{multline*}
\end{proof}

Our final supporting result characterizes short-term changes in the rate of events at a given node, and follows from simple perturbation arguments. 

\begin{lemma}
\label{lemma:lambda_first_order_perturbation}
Let $t, \epsilon \ge 0$ and $i \in [n]$. Then
\begin{multline*}
\left| \lambda_i(t + \epsilon) - \lambda_i(t) - \sum_{j = 1}^n w_{ij} ( N_j(t + \epsilon) - N_j(t)) \right| \\
\le \hspace{-0.1cm} \left( \hspace{-0.1cm} \mu_{\max} + \lambda_i(t) + \hspace{-0.1cm} \sum_{j = 1}^n w_{ij} ( N_j(t + \epsilon) - N_j(t) ) \hspace{-0.1cm} \right) \hspace{-0.1cm} \cderivative \epsilon.
\end{multline*}
\end{lemma}

\begin{proof}
From the definition of $\lambda_i$, we have 
\begin{align}
%\label{eq:lambdai_diff_decomp}
& \lambda_i(t + \epsilon) - \lambda_i(t) - \sum_{j = 1}^n w_{ij} ( N_j(t + \epsilon) - N_j(t)) \nonumber \\
\label{eq:lambdai_diff_decomp_1}
& \hspace{0.2cm} = \mu_i(t+ \epsilon) - \mu_i(t) \\
\label{eq:lambdai_diff_decomp_2}
& \hspace{0.5cm} + \sum_{j = 1}^n w_{ij} \int_t^{t + \epsilon} ( \phi_{ij} (t + \epsilon, x)  - 1) dN_j(x)  \\
\label{eq:lambdai_diff_decomp_3}
& \hspace{0.5cm} + \sum_{j = 1}^n w_{ij} \int_0^t ( \phi_{ij}(t + \epsilon, x) - \phi_{ij}(t , x) ) dN_j(x) .
\end{align}
The first term on the right hand side is at most $\cderivative \mu_{\max} \epsilon$, since the Lipschitz constant of $\mu_i$ is at most $\cderivative \mu_{\max}$ under Assumption \ref{as:baseline}.
To bound \eqref{eq:lambdai_diff_decomp_2}, we note that for any $i,j \in [n]$ and $y \in [x,x+\epsilon]$, $| \phi_{ij}(y,x) - 1 | \le \cderivative \phi_{\max} \epsilon$, where we use the fact that the Lipschitz constant of $\phi_{ij}(\cdot, x)$ is at most $\cderivative$ under Assumptions \ref{as:delay} and \ref{as:smoothness}. Consequently, we can bound
\begin{multline*}
\left| \sum_{j = 1}^n w_{ij} \int_t^{t + \epsilon} ( \phi_{ij}(t + \epsilon , x) - 1 ) dN_j(x) \right| \\
\le \cderivative \phi_{\max} \epsilon \sum_{j = 1}^n w_{ij} ( N_j(t + \epsilon) - N_j(t) ).
\end{multline*}
We next bound \eqref{eq:lambdai_diff_decomp_3}. To this end, we note that
\begin{multline*}
| \phi_{ij}(t + \epsilon , x) - \phi_{ij}(t , x) |\\
\le \int_{t }^{t + \epsilon } \left | \frac{d}{dy} \phi_{ij}'(y,x) \right  | dy \le \cderivative \epsilon \phi_{ij}(t, x),
\end{multline*}
where the final inequality again uses Assumption \ref{as:delay} as well as the fact that $\phi_{ij}(\cdot, x)$ is decreasing. Consequently, we can bound
\begin{multline*}
\left| \sum_{j = 1}^n w_{ij} \int_0^t ( \phi_{ij}(t + \epsilon , x) - \phi_{ij}(t, x) ) dN_j(x) \right| \\
\le \cderivative \epsilon \sum_{j = 1}^n \int_0^t \phi_{ij}(t , x) dN_j(x) \le \cderivative  \lambda_i(t) \epsilon.
\end{multline*}
Above, the second inequality uses the positivity of $\mu_i(t)$.
Putting together all the derived inequalities proves the claimed result. 
\end{proof}

\subsection{Bounding the maximum rate: Proof of Lemma \ref{lemma:maximum_rate}}
\label{subsec:maximum_rate}

For $i,j \in [n]$ and $t \ge 0$, we define the martingale
\[
M_{ij}(t) : = \int_0^t \phi_{ij}(t , s) dN_j(s) - \int_0^t \phi_{ij}(t , s) \lambda_j(s) ds.
\]
We additionally define the stopping time 
\[
\tau : = \inf \{ s \ge 0: \lambda_{\max}(s) \ge \Lambda \}.
\]
Our first result establishes a tail bound for the stopped martingale $M_{ij}(t \land \tau)$.

\begin{lemma}
\label{lemma:Mij_tail_bound}
For any $\Lambda, x,t \ge 0$, it holds that
\[
\p \left( |M_{ij}(t \land \tau ) | \ge x \right) \le 2 \exp \left( - \frac{x^2 / 2}{ \Phi \Lambda + x / 3} \right),
\]
where $\Phi$ is the constant defined in Assumption \ref{as:stability}.
%$\Phi_{ij} : = \sup_{s \ge 0} \int_s^\infty \phi_{ij}(t, s) dt$.
\end{lemma}

To prove the lemma, we will make use of the following modification of Freedman's inequality, due to Shorack and Wellner \cite[Appendix B]{shorack_wellner}.

\begin{lemma}[Freedman's inequality for continuous-time martingales]
\label{lemma:freedman}
Suppose that $\{ Y(t) \}_{t \ge 0}$ is a real-valued continuous-time martingale adapted to the filtration $\{ \cF_t \}_{t \ge 0}$. Suppose further that $Y(t)$ has jumps that are bounded by $C > 0$ in absolute value, almost surely.
For $t \ge 0$, define the predictable quadratic variation $\langle Y \rangle_t$ to be the unique predictable process so that $Y(t)^2 - \langle Y \rangle_t$ is a martingale.
Then for any $x \ge 0$ and $\sigma > 0$, it holds that 
\begin{multline*}
\p ( Y(t) - Y(0) \ge x \text{ and } \langle Y \rangle_t \le \sigma^2 \text{ for some $t \ge 0$} ) \\
\le \exp \left( - \frac{x^2 / 2}{\sigma^2 + Cx / 3} \right).
\end{multline*}
\end{lemma}

We now apply Freedman's inequality to prove Lemma \ref{lemma:Mij_tail_bound}.

\begin{proof}[Proof of Lemma \ref{lemma:Mij_tail_bound}]
For $i,j \in [n]$ and $t \ge 0$, define the process 
\begin{align*}
\widetilde{M}_{ij,t}(s) & : = \int_0^{s \land \tau} \phi_{ij}(t, x) dN_j(x) \\
& \hspace{1cm} - \int_0^{s \land \tau} \phi_{ij}(t , x) \lambda_j(x) dx,
\end{align*}
which is a well-defined martingale for $0 \le s \le t$. Note further that $\widetilde{M}_{ij,t}(t) = M_{ij}(t \land \tau)$, so it suffices to establish tail bounds for $\widetilde{M}_{ij,t}$.

From the definition of $\widetilde{M}_{ij,t}$, it is clear that it jumps by at most $\phi_{ij}(t,t) = 1$. Additionally, for any $0 \le s < t$, we have 
\begin{align}
& \limsup_{\delta \to 0} \frac{\E \left[ \left. ( \widetilde{M}_{ij,t}(s + \delta) - \widetilde{M}_{ij,t}(s) )^2 \right \vert \cF_s \right]}{\delta} \nonumber \\
\label{eq:Mij_limsup_expansion}
& \stackrel{(a)}{=} \hspace{-0.05cm} \limsup_{\delta \to 0} \frac{1}{\delta} \E \left[ \left. \left( \int_{s \land \tau}^{ (s+ \delta)\land \tau} \phi_{ij}(t , x) dN_j(x) \right)^2 \right \vert \cF_s \right].
\end{align}
Above, $(a)$ holds since the sample paths of $\widetilde{M}_{ij,t}$ are \cadlag.
To analyze the integral on the second line further, note that if $s \ge \tau$, it is zero, and if $s < \tau$, we have, due to the definition of $\tau$, the stochastic domination 
\[
\int_{s \land \tau}^{(s + \delta) \land \tau} \phi_{ij} (t , x) dN_j(x) \preceq (\phi_{ij}(t , s) + \eta) Z, 
\]
where $Z \sim \Poi(\delta \Lambda)$ and $\eta = O(\delta)$ by the continuity of $\phi_{ij}(t, \cdot)$ (see Assumption \ref{as:delay}).
It follows immediately that the expressions in \eqref{eq:Mij_limsup_expansion} are at most $\phi_{ij}(t , s )^2 \Lambda$. 
Consequently, the quadratic variation of $\widetilde{M}_{ij,t}$ can be bounded as
\begin{align*}
\langle \widetilde{M}_{ij,t} \rangle_s & \le \left( \int_0^s \phi_{ij}(t ,x)^2 dx \right) \Lambda \\
& \le \left( \int_0^s \phi_{ij}(t , x) dx \right) \Lambda \le \Phi \Lambda,
\end{align*}
where, in the second inequality we have used $\phi_{ij}(t , x) \le \phi_{ij}(x,x ) = 1$.
Applying Lemma \ref{lemma:freedman} to $\widetilde{M}_{ij,t}$ proves the claim.
\end{proof}

Next, we show that with high probability, $M_{ij}$ does not significantly change over short periods of time. 

\begin{lemma}
\label{lemma:Mij_gaps}
Define the set of time indices 
\[
\mathbb{T} : = \left \{ \frac{k}{2 \Lambda} : k \in \mathbb{Z} \cap [0, 2 \Lambda T] \right \}.
\]
There is a constant $C > 0$ depending on $\cderivative, w_{\min}$ and $\Phi$ such that, with probability tending to 1 as $T \Lambda n \to \infty$, it holds for all $t \in \mathbb{T}$ and $i,j \in [n]$ that
\begin{equation*}
 \sup_{s \in [t, t + 1/(2 \Lambda)]} | M_{ij}(s \land \tau) - M_{ij}(t \land \tau) | \le C \log (T \Lambda n).
\end{equation*}
\end{lemma}

\begin{proof}
Let $s \ge t$. From the definition of $M_{ij}$, we can bound
\begin{align}
\label{eq:Mij_diff_bound1}
& | M_{ij}(s \land \tau) - M_{ij}(t \land \tau) |  \le \left( N_j(s \land \tau) - N_j(t \land \tau) \right) \\
\label{eq:Mij_diff_bound2}
& \hspace{0.5cm} + \int_0^{t \land \tau} \left| \phi_{ij}(s \land \tau,x) - \phi_{ij}(t \land \tau,x) \right| dN_j(x) \\
\label{eq:Mij_diff_bound3}
& \hspace{0.5cm} + \int_{t \land \tau}^{s \land \tau} \phi_{ij}(s \land \tau , x) \lambda_j(x) dx 
\\
\label{eq:Mij_diff_bound4}
& \hspace{0.5cm} + \int_0^{t \land \tau} | \phi_{ij}(s \land \tau,x) -  \phi_{ij}(t \land \tau , x) | \lambda_j(x) dx.
\end{align}
We further simplify several terms on the right hand side. 
For all $t \le s \le t + 1/ ( 2 \Lambda)$, we have the simple bound $N_j(s \land \tau) \le N_j( (t + 1 / ( 2 \Lambda)) \land \tau)$.
We can bound \eqref{eq:Mij_diff_bound2} as follows:
\begin{align}
 \int_0^{t \land \tau} & | \phi_{ij}(s \land \tau,x) - \phi_{ij}(t \land \tau,x) | dN_j(x) \\
& \stackrel{(a)}{\le} \cderivative  \int_0^{t \land \tau} \int_{t \land \tau}^{s \land \tau } | \phi_{ij}(r,x) | dr dN_j(x) \nonumber \\
& \stackrel{(b)}{\le} \cderivative (s - t) \int_0^{t \land \tau} \phi_{ij}(t \land \tau ,x) dN_j(x) \nonumber \\
& \stackrel{(c)}{\le} \frac{\cderivative (s - t) \lambda_{i}(t \land \tau)}{w_{ij}} \nonumber \\
\label{eq:Mij_diff_term_2}
& \stackrel{(d)}{\le} \frac{\cderivative}{2 w_{ij}}.
\end{align}
Above, $(a)$ follows since the partial derivative of $\phi_{ij}(s, x)$ with respect to the first coordinate is at most $\cderivative \phi_{ij}(t, x)$ by Assumption \ref{as:smoothness} and the decreasing property of $\phi_{ij}( \cdot, x)$;
$(b)$ uses the fact that $\phi_{ij}(\cdot, x)$ is decreasing;
$(c)$ follows from the definition of $\lambda_i$;
$(d)$ uses the fact that $s - t \le 1 / (2 \Lambda)$ and that $\lambda_i(t \land \tau) \le \Lambda$.

We now bound \eqref{eq:Mij_diff_bound3}. Since $\lambda_j(x) \le \Lambda$ for $t \land \tau \le x \le s \land \tau$, and since $\phi_{ij} \le 1$, we have
\[
\int_{t \land \tau}^{s \land \tau} \phi_{ij}(s \land \tau , x) \lambda_j(x) dx \le (s - t) \Lambda \le \frac{1}{2}, 
\]
where we have used $s - t \le 1 / (2 \Lambda)$ in the final inequality.

We now bound \eqref{eq:Mij_diff_bound4}. We have
\begin{align*}
& \int_0^{t \land \tau} | \phi_{ij}(s \land \tau , x) - \phi_{ij}(t \land \tau , x) | \lambda_j(x) dx \\
& \stackrel{(e)}{\le} \cderivative (s - t) \int_0^{t \land \tau} \phi_{ij}(t \land \tau , x) \lambda_j(x) dx \\
& \stackrel{(f)}{\le} \frac{\cderivative}{2} \int_0^{t \land \tau}\phi_{ij}(t \land \tau , x)  dx \\
& \le \frac{\cderivative \Phi}{2}.
\end{align*}
Above, $(e)$ mirrors the derivation of inequality $(b)$ in \eqref{eq:Mij_diff_term_2}, and
$(f)$ uses $\lambda_j(x) \le \Lambda$ for $x \le t \land \tau$ and $s - t \le 1 / (2 \Lambda)$.

Putting everything together shows that 
\begin{multline}
\label{eq:Mij_diff_v2}
\sup_{t \le s \le t + 1 / (2 \Lambda) } | M_{ij}(s \land \tau) - M_{ij}(t \land \tau) | \\
\le  N_j \left( \left( t + \frac{1}{2 \Lambda} \right) \land \tau \right) - N_j(t \land \tau) 
+ O \left( \frac{\cderivative \Phi}{w_{ij}} \right).
\end{multline}
To complete the analysis of $M_{ij}$, it remains to study $N_j((t + 1/(2 \Lambda)) \land \tau) - N_j(t \land \tau)$. Due to the definition of $\tau$, it is readily seen that this quantity is stochastically bounded by $A \sim \mathrm{Poi}( 1/2)$. Consequently, for $x \ge 0$,
\begin{align*}
& \p \left ( N_j \left( \left( t + \frac{1}{2 \Lambda} \right)  \land \tau \right) - N_j(t \land \tau) \ge x \right ) \\
& \le \p ( A \ge x ) = \sum_{k = \lceil x \rceil}^\infty \frac{e^{-1/2}}{k!} \left( \frac{1}{2} \right)^k \le 2 \left( \frac{1}{2} \right)^x.
\end{align*}
The probability bound above, along with \eqref{eq:Mij_diff_v2} implies that with probability $1 - o ( ( T \Lambda n^2)^{-1})$, 
\[
\sup_{t \le s \le t + 1 / (2 \Lambda)} | M_{ij}(s \land \tau) - M_{ij} ( t \land \tau) | \le C \log ( T \Lambda n),
\]
where $C > 0$ is a sufficiently large constant depending on $\cderivative, \Phi$ and $w_{\min}$ (recall that $w_{\min} : = \min_{i,j \in [n]} w_{ij} \land 1$). Taking a union bound over $t \in \mathbb{T}$ and $i,j \in [n]$ proves the claim.
\end{proof}

We next bound the maximum value of $M_{ij}(s \land \tau)$.

\begin{lemma}
\label{lemma:Mij_maximal}
Let $T, \Lambda \ge 0$. There is a constant $C$ depending on $\cderivative, w_{\min}, \Phi$ such that if $\Lambda \ge C$, then 
\[
\max_{i,j \in [n]} \sup_{0 \le t \le T} | M_{ij} ( t \land \tau ) | \le 2 \sqrt{\Lambda} \log (n T \Lambda),
\]
with probability tending to 1 as $n T \Lambda \to \infty$.
\end{lemma}

\begin{proof}
In this proof, little $o$ notation is with respect to $n T \Lambda$ tending to infinity.

Recall the set of indices $\mathbb{T}$ from the statement of Lemma \ref{lemma:Mij_gaps}.
By Lemma \ref{lemma:Mij_gaps}, it holds with probability $1 - o(1)$ that for all $i,j \in [n]$,  
\begin{equation}
\label{eq:Mij_sup_bound}
\sup_{0 \le t \le T} | M_{ij}(t \land \tau) | \le \max_{t \in \mathbb{T}} | M_{ij}(t \land \tau) | + C_1 \log ( T \Lambda n),
\end{equation}
where $C_1 > 0$ is a sufficiently large constant depending on $\cderivative, w_{\min}, \Phi$. 
Next, Lemma \ref{lemma:Mij_tail_bound} along with a union bound over $t \in \mathbb{T}$ shows that 
\begin{multline*}
\p \left( \max_{t \in \mathbb{T}} |M_{ij}(t \land \tau)| \ge \sqrt{\Lambda} \log ( n T \Lambda) \right) = o \left( \frac{1}{n^2} \right),
\end{multline*}
provided $n T \Lambda \to \infty$ and $\Lambda$ is sufficiently large. Additionally taking a union bound over $i,j \in [n]$ shows that with probability $1 - o(1)$, it holds for all $i,j \in [n]$ that $\max_{t \in \mathbb{T}} | M_{ij}(t \land \tau) | \le \sqrt{\Lambda} \log ( n T \Lambda)$. This result combined with \eqref{eq:Mij_sup_bound} proves the claim.
\end{proof}

We are now ready to put everything together to prove the main result of this subsection.

\begin{proof}[Proof of Lemma \ref{lemma:maximum_rate}]
Recall the stopping time $\tau$. For $t \in [0,T]$ and $i \in [n]$, we can bound
\begin{align*}
\lambda_i(t \land \tau) & = \mu_i(t \land \tau) \\
& \hspace{0.5cm} + \sum_{j = 1}^n w_{ij} \int_0^{t \land \tau} \phi_{ij}(t \land \tau , s) \lambda_j(s) ds \\
& \hspace{0.5cm} + \sum_{j = 1}^n w_{ij} M_{ij}(t \land \tau) \\
& \le \mu_{\max} + \left( \sum_{j = 1}^n w_{ij} \int_0^{t \land \tau} \phi_{ij}(t \land \tau, s) ds \right) \Lambda \\
& \hspace{0.5cm} + d w_{\max} \max_{i,j \in [n]}  M_{ij} ( t \land \tau)  \\
& \le \mu_{\max} + (1 - \eta) \Lambda + 2 d w_{\max} \sqrt{\Lambda} \log ( n T \Lambda),
\end{align*}
where $\eta > 0$ is defined in Assumption \ref{as:stability},
and the final inequality holds with probability $1 - o(1)$ as $n T \Lambda \to \infty$, by Lemma \ref{lemma:Mij_maximal}.
The upper bound derived does not depend on the choice of $t$ or $i$, so
\begin{multline*}
\sup_{0 \le t \le T} \lambda_{\max}(t \land \tau) \le (1 - \eta) \Lambda \\
+ \mu_{\max} + 2d w_{\max} \sqrt{\Lambda} \log ( n T \Lambda).
\end{multline*}
If $\tau \le T$, then $\sup_{0 \le t \le T} \lambda_{\max}(t \land \tau) = \Lambda$, which implies
\[
\Lambda \le \frac{1}{\eta} \left( \mu_{\max} + 2 d w_{\max} \sqrt{\Lambda} \log (n T \Lambda) \right).
\]
However, this inequality does not hold if $\Lambda = d^2 \log^4 (nT)$, for $nT$ sufficiently large. Consequently, it must be the case that $\tau > T$ for this choice of $\Lambda$, hence $\sup_{0 \le t \le T} \lambda_{\max}(t) \le d^2 \log^4 (nT)$ as claimed.
\end{proof}

\subsection{Analysis of short waiting times: Proof of Lemma \ref{lemma:EX}}
\label{sec:short_waiting_times}

The following result, which analyzes the distribution of the first $i$-event after time $t$, is an important building block for Lemma \ref{lemma:EX}.

\begin{lemma}
\label{lemma:Ti_probability}
Let $t \ge 0$ and $\epsilon \in (0,1]$. 
Assume that $\Lambda d \epsilon \le 1/C_1$, $\epsilon \ge n^{-1/2}$, and that $\Lambda \ge C_2 d \log n$, where $C_1,C_2$ are sufficiently large constants depending on $\mu_{\max}$ and $w_{\max}$.
Then there is a constant $C_3 = C_3 ( \cderivative, \mu_{\max})$ such that if $\lambda_{\max}(t) \le \Lambda$,
\[
\left| \p (  N_i[t , t + \epsilon] = 1 \vert \cF_t) - \epsilon \lambda_i(t) \right|  \\
\le C_3 (d \Lambda \epsilon)^2.
\]
\end{lemma}

\begin{proof}
We find it easier to analyze the probability of the event $\{ T_i(t) \le \epsilon \}$. Provided $\lambda_{\max}(t) \le \Lambda$, this is related to the probability of interest via 
\begin{align}
& | \p ( T_i(t) \le \epsilon \vert \cF_t) - \p ( N_i[t, t + \epsilon] = 1 \vert \cF_t ) | \nonumber \\
& \le \p ( \text{At least 2 events at $i$ occur in $[t, t + \epsilon]$} \vert \cF_t) \nonumber \\
\label{eq:probability_relation}
& \le (8 \Lambda \epsilon)^2,
\end{align}
where the final inequality is due to Lemma \ref{lemma:multiple_firings}.

To analyze $\p ( T_i(t) \le \epsilon \vert \cF_t)$, we start by introducing the event 
\[
\cE : = \{\text{No event occurs in $\cN(i) \setminus \{i \}$ in $[t, t + \epsilon]$} \}.
\]
Notice that on $\cE$, the distribution of $T_i(t) - t$ is not influenced by any other events in the Hawkes process.
Hence, we can couple $T_i(t) - t$ with another random variable $T$, where $T$ is the first event time of an independent, inhomogeneous point process with rate function 
\[
r(s) : = \mu_i(t + s) + \sum_{j = 1}^n w_{ij} \int_0^t \phi_{ij}(t + s , x) dN_j(x).
\]
In particular, $r(0) = \lambda_i(t)$ and if $\lambda_{\max}(t) \le \Lambda$ then
\begin{equation}
\label{eq:r_bound}
r(s) \le \mu_{\max} + \lambda_i(t) \le 2 \mu_{\max} \Lambda,
\end{equation}
where the first inequality holds since $\phi$ is decreasing, and the second inequality holds since $\mu_{\max} \ge 1$ and $\Lambda \ge 1$.

Turning to the probability of interest, we have that
\begin{align*}
\p ( \{ T_i(t) \le t + \epsilon \} \cap \cE \vert \cF_t) & \stackrel{(a)}{=} \p ( \{ T \le \epsilon \} \cap \cE \vert \cF_t ) \\
& \stackrel{(b)}{=} \p ( T \le \epsilon \vert \cF_t) \cdot \p ( \cE \vert \cF_t ),
\end{align*}
where $(a)$ follows from the coupling of $T_i(t) - t$ and $T$. The equality $(b)$ holds since $T$ is independent of the Hawkes process and associated events. 

Next, we characterize the two product terms on the right hand side further. By \eqref{eq:r_bound} and Lemma \ref{lemma:T_deterministic_rate}, we have
\begin{equation}
\label{eq:T_characterization}
| \p ( T \le \epsilon \vert \cF_t ) - \epsilon \lambda_i(t) | \mathbf{1}( \lambda_{\max}(t) \le \Lambda) \le 4 \cderivative \mu_{\max}^2 \Lambda^2 \epsilon^2. 
\end{equation}
Additionally, by Lemma \ref{lemma:multiple_firings},
\begin{equation}
\label{eq:E2_characterization}
| \p ( \cE  \vert \cF_t ) - 1 | \mathbf{1}( \lambda_{\max}(t) \le \Lambda ) \le 8d \Lambda \epsilon.
\end{equation}
Together, \eqref{eq:T_characterization} and \eqref{eq:E2_characterization} show that if $\lambda_{\max}(t) \le \Lambda$, then
\begin{equation}
\label{eq:Ti_analysis_2}
| \p ( \{ T_i(t) \le t + \epsilon \} \cap \cE \vert \cF_t ) - \epsilon \lambda_i(t) | \le C d \Lambda^2 \epsilon^2,
\end{equation}
for a constant $C > 0$ depending on $\cderivative$ and $\mu_{\max}$.
On the other hand, if the event $\{ T_i(t) \le \epsilon \} \cap \cE^c$ holds, then at least two events occur in $\cN(i) \cup \{i \}$ in the time period $[t, t + \epsilon]$. By Lemma \ref{lemma:multiple_firings}, the probability of this event is at most $(8 (d + 1) \Lambda \epsilon)^2$. This bound combined with \eqref{eq:probability_relation} and \eqref{eq:Ti_analysis_2} prove the claim.
\qed
\end{proof}

We are now ready to prove the main result of this section.

\begin{proof}[Proof of Lemma \ref{lemma:EX}]
For a collection of (not necessarily distinct) nodes $i_1, \ldots, i_k \in [n]$, we define the indicator variable 
\begin{equation}
X_{i_1 \ldots i_k} (t, \epsilon) : = \prod_{\ell = 1}^k  \mathbf{1} ( N_{i_\ell}[t + (\ell-1)\epsilon, t + \ell \epsilon] = 1 ).
\end{equation}
We prove the more general result that, for any $k \ge 1$, there is a sufficiently large constant $C > 0$ depending on $k, \mu_{\max}, w_{\max}$ and $\cderivative$ such that if $d \Lambda \epsilon \le 1 / C, \epsilon \ge n^{- 1/(k + 1)}$ and $\Lambda \ge Cd \log n$, then 
\begin{multline*}
\left| \E [ X_{i_1 \ldots i_k}(t, \epsilon) \vert \cF_t] - \epsilon^k \prod_{\ell = 1}^k \left( \lambda_{i_\ell}(t) + \sum_{m = 1}^{\ell - 1} w_{i_\ell i_m} \right) \right|\\
\le C (d \Lambda \epsilon)^{k + 1}.
\end{multline*}

We prove the claim by induction on $k$. The base case $k  = 1$ is proved in Lemma \ref{lemma:Ti_probability}, so we proceed by proving the inductive step. In what follows, we assume that all introduced constants are positive, sufficiently large, and depend only on $k, \mu_{\max},w_{\max}$ and $\cderivative$.

Let $k \ge 1$ and $i_1, \ldots, i_{k + 1} \in [n]$ be a collection of (not necessarily distinct) nodes. As a shorthand, we will write $X_{k + 1}$ instead of $X_{i_1, \ldots, i_{k + 1}}(t, \epsilon)$ and $X_k$ instead of $X_{i_1, \ldots, i_k}(t, \epsilon)$.
Define the set $S : = \bigcup_{\ell = 1}^{k + 1} ( \cN(i_\ell) \cup \{ i_\ell \} )$ as well as the events
\[
\cA_m : = \{ \text{At most $m$ events occur in $S$ in $[t, t + m\epsilon]$} \},
\]
where we will take $m \in \{ k, k + 1 \}$. Additionally, we define the event
\[
\cB  : = \left \{ \sup_{t \le s \le t + k\epsilon} \lambda_{\max}(s) \le 2 \Lambda \right \}.
\]
Then we can bound
\begin{align}
\E [ X_{k + 1} \vert \cF_t] & \le \E [ X_{k + 1} \mathbf{1}( \cA_k \cap \cB) \vert \cF_t] \nonumber \\
\label{eq:EX_bound1}
& \hspace{1cm} + \p ( \cA_{k + 1}^c \vert \cF_t) + \p ( \cB^c \vert \cF_t ).
\end{align}
On the event $\{ \lambda_{\max}(t) \le \Lambda \}$, Lemmas \ref{lemma:short_term_rate_bound} and \ref{lemma:multiple_firings} imply that the sum of the second and third terms on the right hand side of \eqref{eq:EX_bound1} is $C_1 ( d \Lambda \epsilon)^{k + 2}$ when $\epsilon^{ k + 2} \ge 1/n$. To handle the first term on the right hand side of \eqref{eq:EX_bound1}, we note that $X_{k + 1} \mathbf{1} ( \cA_k \cap \cB)$ can be written as a product of indicators of the events $\{ X_k = 1 \} \cap \cA_k \cap \cB$ and 
\[
\cD : = \{ N_{i_{k + 1}}[t + k \epsilon, t + (k + 1) \epsilon] = 1 \},
\]
where the former event is $\cF_{t + k \epsilon}$-measurable. In particular, if the former event holds, then
\begin{align}
 \p ( \cD \vert \cF_{t + k \epsilon} ) & \stackrel{(a)}{\le} \epsilon \lambda_{i_{k + 1}}(t + k \epsilon) + C_2 ( d \Lambda \epsilon)^2 \nonumber \\
\label{eq:Tik+1_bound}
&  \stackrel{(b)}{\le} \epsilon \left( \lambda_{i_{k + 1}}(t) + \sum_{\ell = 1}^k w_{i_{k+1} i_\ell} \right) \hspace{-0.1cm} + C_3 ( d \Lambda \epsilon)^2,
\end{align}
where $(a)$ follows from Lemma \ref{lemma:Ti_probability} and $(b)$ is due to Lemma \ref{lemma:lambda_first_order_perturbation}. It follows that
\begin{align*}
& \E [ X_{k + 1} \mathbf{1}( \cA_k \cap \cB) \vert \cF_t] \\
& \hspace{0.5cm} = \E [ \p ( \cD \vert \cF_{t + k \epsilon}) X_k \mathbf{1}( \cA_k \cap \cB) \vert \cF_t ] \\
& \hspace{0.5cm} \stackrel{(c)}{\le} \left( \epsilon \left( \lambda_{i_{k + 1}}(t) + \sum_{\ell = 1}^k w_{i_{k+1} i_\ell} \right) + C_3 ( d \Lambda \epsilon)^2 \right) \\
& \hspace{1cm} \cdot \E [ X_k \mathbf{1}( \cA_k \cap \cB) \vert \cF_t] \\
& \hspace{0.5cm} \stackrel{(d)}{\le} \left( \epsilon \left( \lambda_{i_{k + 1}}(t) + \sum_{\ell = 1}^k w_{i_{k+1} i_\ell} \right) + C_3 ( d \Lambda \epsilon)^2 \right) \\
& \hspace{0.75cm} \cdot \left( \epsilon^k \prod_{\ell = 1}^k \left( \lambda_{i_\ell}(t) + \sum_{m = 1}^{\ell - 1} w_{i_\ell i_m} \right) + C_4 ( d \Lambda \epsilon)^{k + 1} \right) \\
&  \hspace{0.5cm} \stackrel{(e)}{=} \epsilon^{k + 1} \prod_{\ell = 1}^{k + 1} \left( \lambda_{i_\ell}(t) + \sum_{m = 1}^{\ell - 1} w_{i_\ell i_m} \right)  + C_5 ( d \Lambda \epsilon)^{k + 2}.
\end{align*}
Above, $(c)$ holds since the bound in \eqref{eq:Tik+1_bound} is $\cF_{t + k \epsilon}$-measurable. To obtain the inequality $(d)$ we have upper bounded $\E [ X_k \mathbf{1}( \cA_k \cap \cB) \vert \cF_t]$ by $\E [ X_k \vert \cF_t]$ and applied the inductive hypothesis.
Finally, $(e)$ is obtained by consolidating terms.

Although we have only formally proved an upper bound for $\E [ X_{k + 1} \vert \cF_t]$, our analysis is tight up to first order terms in $\epsilon$, and matching lower bounds can be readily obtained. The inductive step holds, therefore the lemma is proved. 
\end{proof}

\subsection{Analysis of key statistics: Proof of Lemma \ref{lemma:D_analysis}}
\label{subsec:D_analysis}

We focus here on proving \eqref{eq:D1_characterization} only, since \eqref{eq:D2_characterization} can be obtained through identical arguments.

As a shorthand, denote $\lambda_{ij}^{(1)}(t) : = w_{ji} \lambda_i(t) - w_{ij} \lambda_j(t)$.
We can then bound
\begin{align}
& \left| D_{ij}^{(1)}(T, \epsilon) - \epsilon^2 \left( w_{ji} \Lambda_i - w_{ij} \Lambda_j \right) \right| \nonumber \\
\label{eq:D_bound_term1}
& \hspace{0.1cm} \le \left| \sum_{k = 0}^{ \lfloor T/(3 \epsilon) \rfloor} ( \Delta_{ij}^{(1)}(k\epsilon, \epsilon) - \E [ \Delta_{ij}^{(1)}(k\epsilon, \epsilon) \vert \cF_{k \epsilon}] \right| \\
\label{eq:D_bound_term2}
& \hspace{0.3cm} + \sum_{k = 0}^{ \lfloor T / (3 \epsilon) \rfloor} \left| \E [ \Delta_{ij}^{(1)}(k\epsilon, \epsilon) \vert \cF_{k \epsilon}] - \epsilon^2 \lambda_{ij}^{(1)}(k \epsilon) \right|.
\end{align}
We proceed by bounding the two terms on the right hand side. Since $\Delta_{ij}^{(1)}(t, \epsilon)$ is at most 1 in absolute value and since none of the $\Delta_{ij}^{(1)}$'s contain information about overlapping intervals, Azuma's inequality along with a union bound implies that with probability $1 - o(1)$, the term \eqref{eq:D_bound_term1} is at most $O( \sqrt{T \log (n) / \epsilon} )$ for all $i,j \in [n]$.

To handle \eqref{eq:D_bound_term2}, we work on the event $\cE : = \{ \sup_{0 \le t \le T} \lambda_{\max}(t) \le d^2 \log^4(nT) \}$, which holds with probability $1 - o(1)$ by Lemma \ref{lemma:maximum_rate}. On this event, each summand in \eqref{eq:D_bound_term2} is $O( d^9 \log^{12} (nT) \epsilon^3)$ by Lemma \ref{lemma:EX}, provided $d^3 \log^4(nT) \epsilon$ is sufficiently small and $\epsilon \ge n^{-1/3}$. Under these conditions, the entire sum in \eqref{eq:D_bound_term2} is $O(d^9 T \log^{12}(nT) \epsilon^2)$. Together, \eqref{eq:D_bound_term1} and \eqref{eq:D_bound_term2} prove \eqref{eq:D1_characterization}. \hfill \qed

\end{document}